\documentclass[]{article}
\usepackage{float,graphicx,amsmath,amsthm,amssymb,subcaption,caption,bbold,lineno,hyperref,mathtools,authblk}
\mathtoolsset{showonlyrefs}
\newtheorem{theorem}{Theorem}
\newtheorem{proposition}{Proposition}
\newtheorem{corollary}{Corollary}

\theoremstyle{definition}

\newtheorem{examplex}{Example}
\newenvironment{example}
{\pushQED{\qed}\examplex}
{\popQED\endexamplex}
\usepackage{xcolor}
\newtheorem{algorithm}{Algorithm}
\newtheorem{remark}{Remark}
\newtheorem{definition}{Definition}
\usepackage[margin=1.6in]{geometry}
\usepackage[normalem]{ulem}

\newcommand{\bx}{\boldsymbol{x}}
\newcommand{\mcA}{\mathcal{A}}
\newcommand{\bv}{\boldsymbol{v}}
\newcommand{\bu}{\boldsymbol{u}}
\newcommand{\bp}{\boldsymbol{p}}
\newcommand{\bbf}{\boldsymbol{f}}
\newcommand{\p}{\partial}
\newcommand{\bff}{\boldsymbol{f}}
\newcommand{\bb}{\boldsymbol{b}}
\newcommand{\bc}{\boldsymbol{c}}
\newcommand{\dd}{\mathrm{d}}
\newcommand{\bg}{\boldsymbol{g}}
\newcommand{\bz}{\boldsymbol{z}}
\newcommand{\bq}{\boldsymbol{q}}
\newcommand{\br}{\boldsymbol{r}}
\newcommand{\diag}{\mathrm{diag}}
\newcommand{\dsone}{\mathbb{1}}
\newcommand{\R}{\mathbb{R}}
\newcommand{\by}{\boldsymbol{y}}

\title{\bf On the preservation of second integrals by Runge-Kutta methods}
\date{\today}
\author{Benjamin K Tapley\thanks{email: \href{mailto:benjamin.tapley@ntnu.no}{benjamin.tapley@ntnu.no}}}
\affil{Department of Mathematical Sciences, The Norwegian University of Science and Technology, 7491 Trondheim, Norway}
\begin{document}
\maketitle
\begin{abstract}
	One can elucidate integrability properties of ordinary differential equations (ODEs) by knowing the existence of second integrals (also known as weak integrals or Darboux polynomials for polynomial ODEs). However, little is known about how they are preserved, if at all, under numerical methods. Here, we show that in general all Runge-Kutta methods will preserve all affine second integrals but not (irreducible) quadratic second integrals. A number of interesting corollaries are discussed, such as the preservation of certain rational integrals by arbitrary Runge-Kutta methods. We also study the special case of affine second integrals with constant cofactor and discuss the preservation of affine higher integrals. 
\end{abstract}

\section{Introduction}
Many ordinary differential equations (ODEs) possess structure that gives rise to certain qualitative behaviors of the exact solution. For example, the Hamiltonian function remains constant along the flow of Hamiltonian systems \cite{leimkuhler2004simulating} or the phase space volume is preserved along the flow of source-free ODEs \cite{Kang1995}. When solving such ODEs numerically, it is important, especially for long-term simulations, that these features are preserved when possible \cite{hairer2006geometric}. The preservation of first integrals is arguably one of the most important properties for a numerical method to inherit and there exist many well designed numerical methods for this purpose such as discrete gradient methods \cite{dahlby2011preserving}. Runge-Kutta methods, on the other hand, will always preserve all linear first integrals \cite{hairer2006geometric} and they preserve all quadratic first integrals when their stability matrix vanishes \cite{cooper1987stability}. Runge-Kutta methods cannot preserve all cubic (or higher) first integrals \cite{hairer2006geometric}, however there exist particular Runge-Kutta methods that can preserve a polynomial Hamiltonian \cite{celledoni2009energy}. The theory behind preserving first integrals when solving ODEs numerically has been the center of many influential studies \cite{quispel1997solving,quispel2008new,hairer2010energy,wu2013efficient}, however, many ODEs also possess \textit{second integrals} \cite{goriely2001integrability}, which are also important for numerical methods to preserve. Second integrals also go by the names \textit{weak integrals}, \textit{special integrals} \cite{albrecht1996algorithms}, \textit{algebraic invariant curves }\cite{hewitt1991algebraic} or \textit{particular algebraic solutions} \cite{jouanolou2006equations}, for example, and will refer to polynomial second integrals of polynomial vector fields as \textit{Darboux polynomials} \cite{ollagnier1993non}. Second integrals are important to study as they behave like first integrals on their zero level sets. These level sets also divide phase space into regions that are qualitatively different (e.g., bounded or unbounded) and represent barriers for transport. Moreover, many first integrals, especially rational or time-dependent first integrals, are constructed by taking certain products and/or quotients of second integrals. Despite their importance, the preservation of second integrals by numerical methods has not been addressed until only very recently \cite{celledoni2019discrete,celledoni2019using}. In this paper we will discuss the second-integral-preserving properties of Runge-Kutta methods. \\


Consider an autonomous ODE in $\mathbb{R}^n$
\begin{equation}\label{ODE1}
\dot{\bx} = \bff(\bx),
\end{equation}
then a \textit{second integral} \cite{goriely2001integrability} of the ODE \eqref{ODE1} is a function $p(\bx)$ that satisfies 
\begin{equation}\label{DP}
\dot{p}(\bx) = \bff(\bx)^T\nabla p(\bx) = c(\bx)p(\bx),
\end{equation}
where the dot denotes $\frac{\mathrm{d}}{\mathrm{d}t}$ and $c(\bx)$ is called the \textit{cofactor} of $p(\bx)$. If $\bff(\bx)$ is polynomial, then $p(\bx)\in\mathbb{K}_m[\bx]$ is referred to as a Darboux polynomial, where $\mathbb{K}_m[\bx]$ is the class of polynomials over the field $\mathbb{K}$ of degree $m$ in the variables $\bx$. The discrete-time analogue of $p(\bx)$, is referred to as a \textit{discrete second integral} of a map $\varphi_h:\R^n\rightarrow\R^n$ (or \textit{discrete Darboux polynomial} when $\varphi_h(\bx)$ is polynomial or rational), which is a function $p(\bx)$ that satisfies
\begin{equation}\label{cofactoreq}
p(\varphi_h(\bx))=\tilde{c}(\bx)p(\bx),
\end{equation}
where $\tilde{c}(\bx)$ is called the \textit{discrete cofactor} of $p(\bx)$. This is a discrete analogue of equation \eqref{DP} and was recently introduced in \cite{celledoni2019discrete,celledoni2019using}. In our context, $\varphi_h(\bx)\approx\bx(h)$ is a numerical method designed to approximate the solution of the ODE \eqref{ODE1} for some time step $h$. The map $\varphi_h$ is said to \textit{preserve} a Darboux polynomial $p(\bx)$ if there exists a discrete cofactor $\tilde{c}(\bx)$ that satisfies \eqref{cofactoreq} and is independent of the choice of basis for $p(\bx)$. The last point is important, otherwise any polynomial $p(\bx)$ is a discrete Darboux polynomial by letting $\tilde{c}(\bx):=p(\varphi_h(\bx))/p(\bx)$.\\ 

In this paper we will mainly consider ODEs with one or more affine second integrals of the form $p(\bx) = \bp^T \bx+r$ where $\bp\in \mathbb{R}^n$ and $r\in\R$. Note that we can take the constant $r=0$ without loss of generality as we will do throughout the paper. In particular, we focus on the case where the map $\varphi_h$ comes about as a Runge-Kutta method applied to an ODE. That is, letting $\varphi_h(\bx)$ denote one step of an $s$-stage Runge-Kutta method applied to \eqref{ODE1} with initial condition $\bx$ and Butcher tableau given by 
\begin{equation}\label{butchertable}
\begin{array}{c|c}
\mathcal{C}&\mcA\\
\hline
~&\bb^T
\end{array}
\end{equation}
then this defines the method
\begin{align}
\bg_i = & \bx + h \sum_{j=1}^{s} a_{ij}\bbf(\bg_j), \quad{\mathrm{for}~i=1,...,s} \label{rk0}\\
\varphi_h(\bx) = & \bx + h \sum_{j=1}^{s} b_{j}\bbf(\bg_j) \label{rk1},
\end{align}
where $\varphi_h(\bx)$ is the Runge-Kutta map with time-step $h$ applied to the initial point $\bx$ and $\mcA=[a_{ij}]$. For explicit Runge-Kutta maps the sum in equation \eqref{rk0} runs from 1 to $i-1$ as $\mcA$ is strictly lower triangular. \\

The paper begins by proving that Runge-Kutta methods cannot preserve irreducible quadratic second integrals. We then prove that Runge-Kutta methods always preserve affine second integrals and discuss the implications of this theorem including the fact that all Runge-Kutta methods preserve all rational integrals that are affine in the numerator and denominator. We then focus on the special case of affine second integrals with constant cofactor. The final section is on the preservation of affine higher integrals. 

\section{The non-preservation of quadratic second integrals}
\begin{theorem}\label{thm:noquad}
	Given an autonomous ODE with an irreducible quadratic Darboux polynomial $p(\bx)\in\mathbb{K}_2[\bx]$, then no Runge-Kutta method can preserve $p(\bx)$. 
\end{theorem}
\begin{proof}
	The proof begins in a similar way to that of \cite[Thm. 2.2]{hairer2006geometric}. Consider a quadratic second integral $p(\bx)=\bx^TQ\bx$ for some symmetric adjoint matrix $Q\in\mathbb{K}^{n\times n}$. Then 
	\begin{equation}
		p(\varphi_h(\bx)) = \bx^TQ\bx + h \sum_{i=1}^{s}b_i\bff(\bg_i)^TQ\bx+h \sum_{j=1}^{s}b_j\bx^TQ\bff(\bg_i) + h^2 \sum_{i,j=1}^{s}b_ib_j\bff(\bg_i)^TQ\bff(\bg_j).
	\end{equation}
	Inserting $\bx  = \bg_i-h\sum_{j=1}^{s}a_{ij}\bff(\bg_j)$ in the two $O(h)$ terms yields
	\begin{equation}
	p(\varphi_h(\bx)) = \bx^TQ\bx + 2h \sum_{i=1}^{s}b_i\bg_i^TQ\bff(\bg_i) + h^2 \sum_{i,j=1}^{s}(b_ib_j-b_ia_{ij}-b_ja_{ji})\bff(\bg_i)^TQ\bff(\bg_j). \label{proof0}
	\end{equation}
	Note that equation \eqref{DP} implies $\bg_i^TQ\bff(\bg_i) = c(\bg_i) p(\bg_i)$. Computing $p(\bg_i)$ gives
	\begin{align}
		 p(\bg_i) = & \bx^TQ\bx + 
		 h \sum_{i=1}^{s}a_{ij}\bff(\bg_j)^TQ\bx+
		 h \sum_{j=1}^{s}a_{ik}\bx^TQ\bff(\bg_k) + 
		 h^2 \sum_{j,k=1}^{s}a_{ij}a_{ik}\bff(\bg_j)^TQ\bff(\bg_k) \\	
		 = & p(\bx) +
		 h \sum_{j=1}^{s}a_{ij}c(\bg_j)p(\bg_j) + \overbrace{ 
		 h^2\sum_{j,k=1}^{s}(a_{ij}a_{ik}-a_{ij}a_{jk}-a_{ik}a_{kj})\bff(\bg_j)^TQ\bff(\bg_k)}^{:=H_i}, \label{proof01}
	\end{align}
	where we have again inserted the expression for $\bx$ and used equation \eqref{DP} to arrive at \eqref{proof01}. Note that this is now a linear system of $s$ equations for $p(\bg_j)$, which we can solve for. Denoting by $\mathbf{P} := (p(\bg_1),...,p(\bg_s))^T\in\R^{s}$ and $\mathbf{H}\in\R^{s}$ the vector whose $i$th component is defined in equation \eqref{proof01} then 
	\begin{align}
		\mathbf{P} & = \mathbb{1}p(\bx) + h\mcA D_c\mathbf{P}+\mathbf{H}\\
		& = (I - h\mcA D_c)^{-1}\left(p(\bx)\mathbb{1}+\mathbf{H}\right),
	\end{align}
	where $D_c:=\diag([c(\bg_1),...,c(\bg_s)])\in\R^{s\times s}$ and $\mathbb{1}\in\R^{s}$ is the ones vector. Equation \eqref{proof0} therefore reads
	\begin{align}
		p(\varphi_h(\bx)) =& p(\bx) + h \bb^TD_c\mathbf{P} + h^2 \sum_{i,j=1}^{s}(b_ib_j-b_ia_{ij}-b_ja_{ji})\bff(\bg_i)^TQ\bff(\bg_j)\\
		=& p(\bx) + h \bb^TD_c(I - h\mcA D_c)^{-1}\left(p(\bx)\mathbb{1}+\mathbf{H}\right) \\
		&\qquad + h^2 \sum_{i,j=1}^{s}(b_ib_j-b_ia_{ij}-b_ja_{ji})\bff(\bg_i)^TQ\bff(\bg_j) \label{proof02}.
	\end{align}
	For the discrete cofactor $\tilde{c}(\bx) = p(\varphi_h(\bx))/p(\bx)$ to be independent of $p(\bx)$, we require that $p(\bx)$ divides equation \eqref{proof02}. This is only true if the following conditions are satisfied 
	\begin{equation}
		b_ib_j-b_ia_{ij}-b_ja_{ji} = 0, \quad a_{ij}a_{ik}-a_{ij}a_{jk}-a_{ik}a_{kj} = 0, \quad \text{for} \quad i,j,k=1,...,s
	\end{equation}
	which is only true when $a_{i,j}=0$ and $b_j=0$ for $i,j=1,...,s$. Therefore no Runge-Kutta method can, in general, preserve irreducible quadratic second integrals. 
\end{proof}

Note that above theorem only applies to irreducible second integrals as Runge-Kutta methods can preserve quadratic second integrals that are the product of affine second integrals. For example, when a quadratic second integral can be expressed as $p(\bx) = \bx^TQ\bx = (\bx^T\bu)(\bv^T\bx)$ where $Q :=  \bu\bv^T$ then the preservation of $p(\bx)$ is equivalent to the preservation of the two linear second integrals $\bx^T\bu$ and $\bx^T\bv$. Due to theorem \ref{thm: lin RK cofactors} in the next section, this is always true for Runge-Kutta methods. \\

Furthermore, in \cite{celledoni2019discrete} many examples of Kahan's method (which is a Runge-Kutta method for quadratic vector fields \cite{celledoni2012geometric}) preserving irreducible quadratic Darboux polynomials are given. This suggests that despite the negative result of theorem \ref{thm:noquad}, there do exist cases where certain Runge-Kutta methods can preserve certain quadratic second integrals. We will leave this for a future study and instead focus now on the preservation of affine second integrals by Runge-Kutta methods. 

\section{The preservation of affine second integrals}
We begin with an example of a planar ODE and its discretisation by a second-order Runge-Kutta method. 
\begin{example}
	Consider the following ODE in two dimensions
	\begin{equation}
	\dot{x}={x}^{2}+2\,xy+3\,{y}^{2},\quad	\dot{y}=2\,y \left( 2\,x+y \right) , \label{ralston}
	\end{equation}
	This ODE was studied in \cite{collins1995algebraic} and has the following three linear Darboux polynomials
	\begin{equation}
	p_1(\bx) = x + y, \quad p_2(\bx) = x - y, \quad p_3(\bx) = y,
	\end{equation}
	that correspond to the cofactors 
	\begin{equation}
	c_1(\bx) = x + 5y, \quad c_2(\bx) = x - y, \quad c_3(\bx) = 4x+2y.
	\end{equation}

		\begin{figure}[h!]
		\centering
		\begin{subfigure}{0.49\textwidth}
			\includegraphics[width=\linewidth]{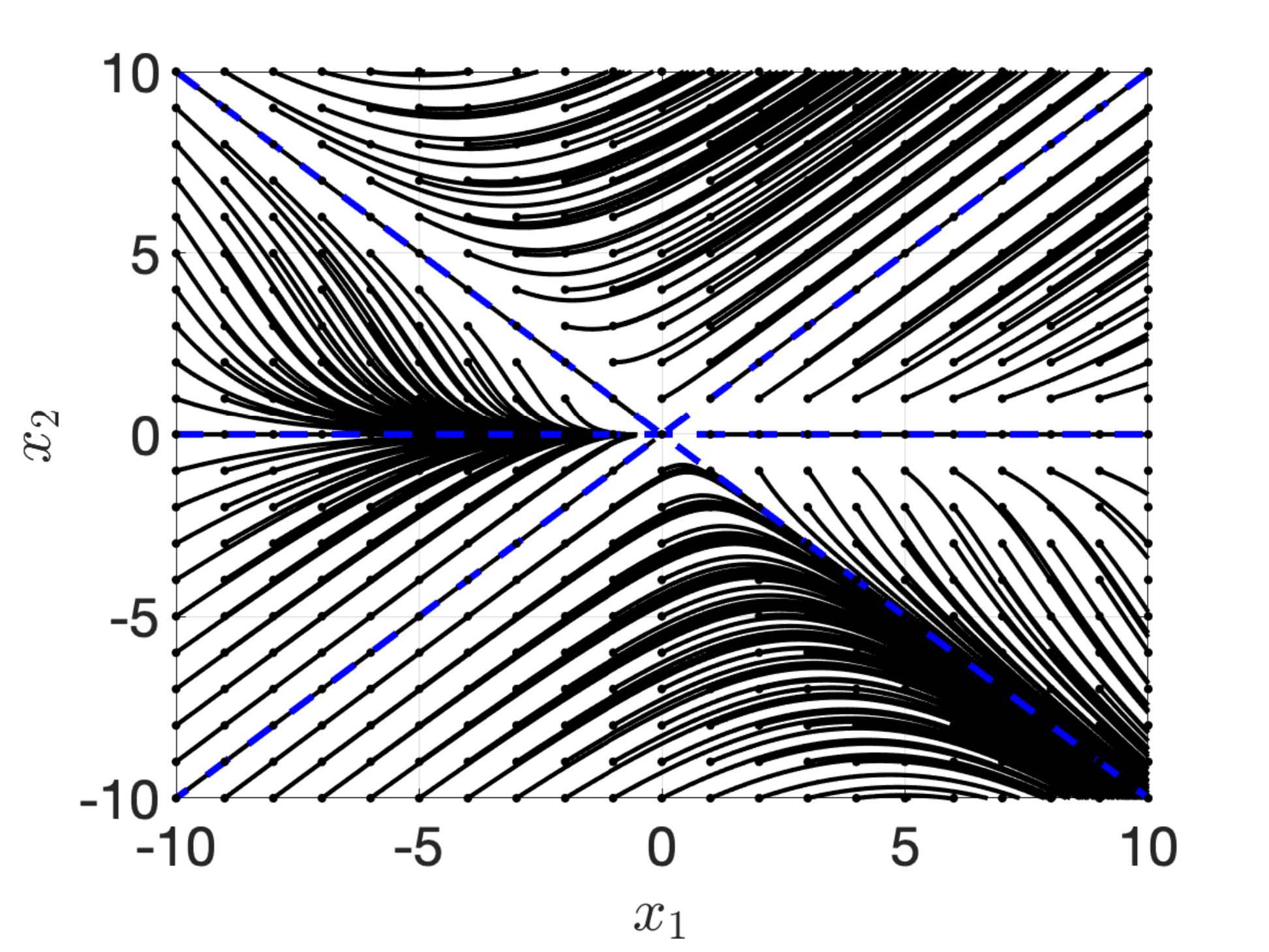}
			\caption{}
			\label{fig:phaselinesralston}
		\end{subfigure}
		\begin{subfigure}{0.49\textwidth}
			\includegraphics[width=\linewidth]{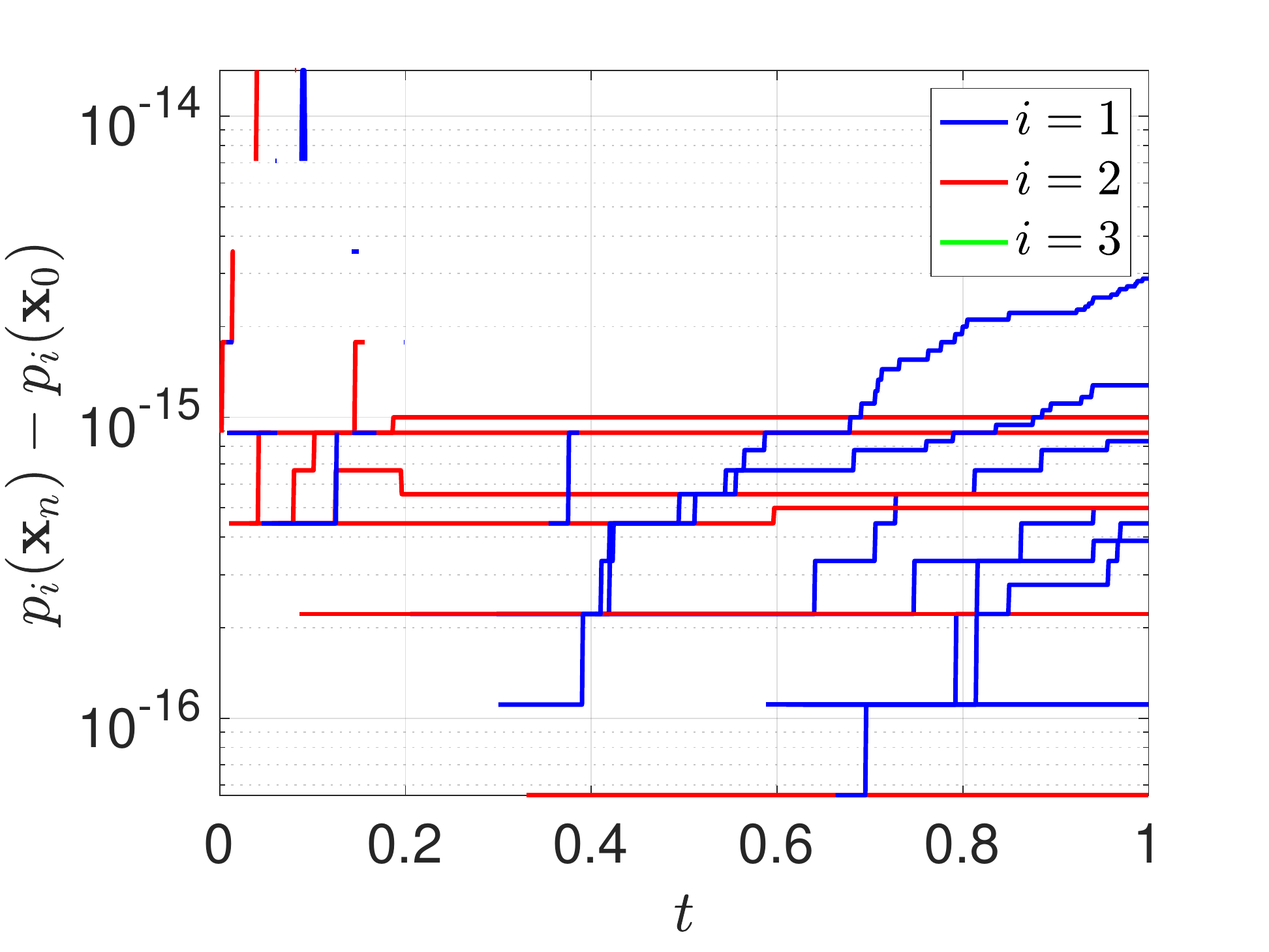}
			\caption{}
			\label{fig:dperrorralston}
		\end{subfigure}
		\caption{The phase portrait of the ODE \eqref{ralston} and the errors of the second integrals $p_i(\bx)$ for initial conditions satisfying $p_i(\bx_0)=0$. Note that $p_3(\bx_n)-p_3(\bx_0)=0$ and therefore does not show on the semi-log axis. The initial conditions are shown by black dots and are located on the grid $(-10 + i,-10+j)$ for $i,j=0,...,20$. }
		\label{fig:ralston}
	\end{figure}
	Consequently, the system also possesses the first integral $p_1(\bx)p_2(\bx)^3p_3(\bx)^{-1}$. The system is discretised using Ralston's method with a time step of $h=0.001$ and the phase portrait on the square $[-10, 10]^2$ is presented in figure \ref{fig:phaselinesralston}. Here, the level sets $p_i(\bx)=0$ for $i=1,2,3$ are represented by blue dashed lines. We see that numerical solutions starting on one of these zero level sets remain on the level set. This is exemplified in figure \ref{fig:dperrorralston} which shows the errors $p_i(\bx_n)-p_i(\bx_0)$ for the numerical solutions starting from $p_i(\bx_0)=0$ and for $i=1,2,3$. Here, we see that the errors are all within machine precision, implying that these three second integrals behave like first integrals on their zero level sets as the exact solution dictates.
\end{example}

Second integrals are important as they divide phase space into sections with qualitatively different behavior. We see that Ralston's method has preserved the Darboux polynomials $p_i(\bx)$, meaning that it produces a qualitatively similar phase portrait. This fact is due to the following theorem. 

\begin{theorem}\label{thm: lin RK cofactors}
	If an autonomous ODE $\dot{\bx}=\bff(\bx)$ possesses an affine second integral $p(\bx) = \bp^T\bx$ for $\bp\in\R^{n}$ with cofactor $c(\bx)$ satisfying $\bp^T\bff=c(\bx)\bp^T\bx$ then a Runge-Kutta map $\varphi_h$ of the ODE possesses the discrete second integral $\bp^T\bx$ that satisfies $\bp^T\varphi_h(\bx) = \tilde{c}(\bx)\,\bp^T\bx$ where the discrete cofactor is given by
	\begin{equation}
	\tilde{c}(\bx) = 1+h\bb^TD_c\,(I-h\mcA D_c)^{-1}\dsone,
	\end{equation}
	where $D_c:=\diag([c(\bg_1),...,c(\bg_s)])\in\R^{s\times s}$ and $\dsone\in\R^n$ is the vector of ones.
\end{theorem}

\begin{proof}
	Let $\varphi_h$ denote the Runge-Kutta map defined by equations \eqref{butchertable}, \eqref{rk0} and \eqref{rk1}. Now let $G := (\bg_1,...,\bg_s)^T $ and $F := (\bff(\bg_i),...,\bff(\bg_s))^T $ denote the $s\times n$ matrices whose $i$'th rows are $\bg_i^T$ and $\bff_i^T$, respectively. Then
	\begin{equation}
	p(\varphi_h(\bx)) = \bp^T\bx + h \sum_{j=1}^{s} b_{j}\bp^T\bbf(\bg_j) 
	= \bp^T\bx + h \bb^TF\bp 
	= \bp^T\bx + h \bb^TD_cG\bp\label{proof1}.
	\end{equation}
	We have for $G\bp$ the following
	\begin{equation}
	G\bp = \dsone_s\bp^T\bx + h\mcA F \bp 
	= \left(I - h \, \mcA D_c\right)^{-1}\dsone_s\bp^T\bx \label{proof2}
	\end{equation}
	due to the fact that $F\bp = D_cG\bp$. Inserting \eqref{proof2} into \eqref{proof1} and dividing by $\bp^T\bx$ we arrive at the desired result
	\begin{equation}
	\frac{\bp^T\varphi_h(\bx)}{\bp^T\bx} = \tilde{c} = 1+h\bb^TD_c\,(I-h\mcA D_c)^{-1}\dsone_s.
	\end{equation}
\end{proof}

This is a generalisation of theorem 1 in \cite{celledoni2019using}, which shows that Kahan's method preserves all affine Darboux polynomials for quadratic vector fields. The discrete cofactor $\tilde{c}(\bx)$ of theorem \ref{thm: lin RK cofactors} depends only on the Butcher table coefficients, the vector field $\bff(\bx)$ and the (continuous) cofactor $c(\bx)$. Furthermore, $\tilde{c}(\bx)$ is in general rational and implicitly defined due the dependence of $D_c$ on the stage values $\bg_i$. However, when $\bff(\bx)$ is polynomial then explicit Runge-Kutta maps applied to $\bff(\bx)$ yield polynomial maps. One would therefore expect $\tilde{c}$ to be known explicitly and be polynomial. \\

\begin{remark}\label{rem:expl cof}
	For all explicit Runge-Kutta methods applied to polynomial ODEs then $\tilde{c}(\bx)$ is polynomial and can be written explicitly. This can be shown by observing that the matrix $I - h \mcA D_c   = I + L$ where $L:=- h\mcA D_c $ is strictly lower triangular and therefore $(I + L)^{-1} = I + \tilde{L}$ where $\tilde{L}$ is also strictly lower triangular. Moreover, as $\det(I+L)=1$, its inverse is equal to its adjugate, that is $(I+L)^{-1}=\mathrm{adj}(I+L) = I + \tilde{L}$ and therefore $\tilde{L}$ is polynomial in the components of $L$. As the cofactor $c(\bx)$ is polynomial for polynomial ODEs, it follows that $\tilde{c}(\bx)$ is polynomial.\\
\end{remark} 

An interesting implication of theorem \ref{thm: lin RK cofactors} is that if an ODE possesses two linear second integrals with the same cofactor, then so does $\varphi_h$. This leads to the following corollary about the preservation of rational integrals.
\begin{corollary}\label{cor:rational integral}
	All Runge-Kutta methods preserve all rational first integrals of the form $H(\bx)=Q(\bx)/R(\bx)$ for $Q(\bx)$ and $R(\bx)$ affine. 
\end{corollary}
\begin{proof}
	Any ODE with a first integral $H(\bx)=Q(\bx)/R(\bx)$ can be written as the following system
	\begin{equation}\label{rat int ode}
	\dot{\bx}=\bff(\bx)=S(\bx)\nabla\left(\frac{Q(\bx)}{R(\bx)}\right)
	\end{equation}
	for some skew-symmetric matrix $S(\bx)=-S(\bx)^T$. Without loss of generality we can let $Q(\bx) = \bq^T\bx$ and $R(\bx) = \br^T\bx$ for constant vectors $\bq,\br\in\R^n$. We now show that these two functions are second integrals of the ODE with the same cofactor by computing their time derivatives
	\begin{equation}
	\frac{\dd}{\dd t} Q(\bx) = \bq^T\bff = \frac{\bq^T S(\bx)\left(\bq^T\bx\,\br - \br^T\bx \,\bq \right)}{{R(\bx)^2}} = \frac{\bq^T S(\bx)\br }{{R(\bx)^2}} Q(\bx)
	\end{equation}	
	due to skew-symmetry of $S(\bx)$. Similarly, 
	\begin{equation}
	\frac{\dd}{\dd t} R(\bx) = \frac{\bq^T S(\bx)\br }{{R(\bx)^2}} R(\bx),
	\end{equation}	
	that is, $Q(\bx)$ and $R(\bx)$ are second integrals with cofactor $\frac{\bq^T S(\bx)\br }{{R(\bx)^2}}$. Due to theorem \ref{thm: lin RK cofactors}, all Runge-Kutta methods preserve these second integrals and they both correspond to the same discrete cofactor. Therefore their quotient is an integral of the Runge-Kutta map. 
\end{proof}
Note that corollary \ref{cor:rational integral} applies to general ODEs. The same statement can me made for polynomial ODEs by scaling \eqref{rat int ode} by $R(x)^a$, $a\ge2$ and $a\in\mathbb{N}$. We now give two examples of explicit Runge-Kutta methods preserving a rational integral. The first is of a simple Lotka-Volterra system, the second is of a non-polynomial vector field.\\

\begin{example}[A Lotka-Volterra system with a rational integral]\label{eg: LV}
	Consider the following 2D Lotka-Volterra system
	\begin{equation}
	\dot{x}=x(x-y),\quad \dot{y}=y(x-y).
	\end{equation}
	Here $x$ and $y$ are clearly Darboux polynomials with cofactor $c = x-y$ implying that $H(\bx)=\frac{x}{y}$ is a first integral of the ODE. Now consider the generic second order explicit Runge-Kutta map $\varphi_h$ defined by the Butcher Tableau 
	\begin{equation}
	\begin{array}{c|cc}
	0 & 0 & 0\\
	\theta & \theta & 0\\
	\hline
	~& 1-\frac{1}{2\theta} & \frac{1}{2\theta} \\
	\end{array}
	\end{equation}
	Note that setting $\theta=1/2,2/3,1$ yields the explicit midpoint, Ralston's and Heun's methods, respectively. Then according to theorem \ref{thm: lin RK cofactors} the discrete cofactor is 
	\begin{equation}
	\tilde{c}(\bx) = 1+{{ \left( \,x_{{1}}
			-\,x_{{2}} \right) h}}+
	\left( x_{{1}}-x_{{2}} \right) ^{2}{h}^{2}+{\frac{\theta}{2}\, \left( x_{{1}}-x_{{2}} \right) ^{3}{h}^{3}}
	\end{equation}
	and is polynomial due to remark \ref{rem:expl cof}.	Indeed, we can show that the Runge-Kutta map satisfies 
	\begin{equation}
	\varphi_h(\bx) = \tilde{c}(\bx)\bx
	\end{equation}
	meaning that $x$ and $y$ are discrete Darboux polynomials of $\varphi_h$ with cofactor $\tilde{c}$. This implies that $H(\bx)=\frac{x}{y}$ is a first integral of $\varphi_h$. 
\end{example}

\begin{example}[A non-polynomial ODE with a rational integral]
	Consider the following ODE
	\begin{equation}
		\dot{x} = \frac{2\,x+\alpha} {\left( x-y \right) ^{{\frac{3}{2}}}},\quad \dot{y} = \frac{2\,y+\alpha} {\left( x-y \right) ^{{\frac{3}{2}}}}
	\end{equation}
	which has the following two second integrals
	\begin{equation}
		p_1(\bx) = x+y+\alpha, \quad p_2(\bx) = x-y
	\end{equation}
	that both correspond to the cofactor 
	\begin{equation}
		c(\bx) = \frac{2}{\left( x-y \right) ^{3/2}}.
	\end{equation}
	This means that $H(\bx)=(x+y+\alpha)/(x-y)$ is a first integral of the ODE. Now apply to this ODE the generic second-order explicit Runge-Kutta map $\varphi_h$ from example \ref{eg: LV}. Then $\varphi_h$ possesses the discrete second integrals $p_1(\bx)$ and $p_2(\bx)$ with cofactor 
		\begin{equation}
			\tilde{c}(\bx) =\frac{ \left( h\sqrt {x-y}+a(\bx) \left(  \left(  \left( x-y \right) ^{{\frac{3}{2}}}+2\,h
				\right) \theta-h \right)  \right)}{a(\bx) \theta \left( x-y \right) ^{{\frac{3}{2}}}}
		\end{equation}
		where 
		\begin{equation}
			a(\bx) = \sqrt {{ \left(  \left( x-y \right) ^{{\frac{1}{2}}}+2\,h\theta \left( x-y \right) ^{-{\frac{1}{2}}}
					\right)}},
		\end{equation}
		in agreement with theorem \ref{thm: lin RK cofactors}, hence $H(\bx)$ is an integral of $\varphi_h$. We can also verify by direct computation that $\varphi_h$ satisfies $H(\varphi_h(\bx))=H(\bx)$. 
\end{example}

If an ODE possesses a rational integral of the form prescribed by corollary \ref{cor:rational integral}, then it is not necessarily a straightforward task to determine its form without knowledge of the cofactor. However, due to the following theorem, one can immediately determine the discrete cofactor corresponding to the numerator and denominator of the rational integral by computing the Jacobian determinant of a Runge-Kutta map applied to the ODE. From this one can infer the (continuous) cofactor due to the fact that $\lim\limits_{h\rightarrow0}((\tilde{c}(\bx)-1)/h)=c(\bx)$

\begin{theorem}\label{thm:det}
	If a Runge-Kutta map $\varphi_h$ possesses two or more affine second integrals with the same cofactor $\tilde{c}(\bx)$, then $\tilde{c}(\bx)$ divides the Jacobian determinant of the map $ \det\left(\frac{\partial \varphi_h(\bx)}{\partial \bx}\right)$. 
\end{theorem}
\begin{proof}
	Without loss of generality we can assume the two affine second integrals are $p_1(\bx)=x_1$ and $p_2(\bx)=x_2$ that both correspond to the cofactor $\tilde{c}(\bx)$. Letting $\varphi_h:(x_1,...,x_n) \mapsto \varphi_h(x_1,...,x_n):=(x_1',...,x_n')$ then the map can be written as 
	\begin{equation}
		\varphi_h\left(\begin{array}{c}
		x_1\\
		x_2\\
		x_3\\
		\vdots\\
		x_n
		\end{array}\right)
		 = 
		\left(\begin{array}{c}
		x_1\tilde{c}(\bx)\\
		x_2\tilde{c}(\bx)\\
		x_3'\\
		\vdots\\
		x_n'
		\end{array}\right).
	\end{equation}
	Taking the Jacobian of $\varphi_h(\bx)$ with respect to $\bx$ gives
	\begin{equation}
{	\renewcommand{\arraystretch}{1.5}
		J_{\varphi_h} := \frac{\partial \varphi_h(\bx)}{\partial \bx} = 
		{\left(\begin{array}{cccc} 
		x_1\frac{\p \tilde{c}}{\p x_1} & \dots & x_1\frac{\p \tilde{c}}{\p x_n} \\
		x_2\frac{\p \tilde{c}}{\p x_1} & \dots & x_2\frac{\p \tilde{c}}{\p x_n} \\
		\frac{\p x'_3}{\p x_1} & \dots & \frac{\p x'_3}{\p x_n}\\
		\vdots && \vdots\\
		\frac{\p x'_n}{\p x_1} & \dots & \frac{\p x'_n}{\p x_n}\\
		\end{array}\right)} + 
		\left(\begin{array}{ccccc}
		\tilde{c}(\bx) & 0 & 0 & \dots &0 \\
		0 &\tilde{c}(\bx) & 0 &   & \\
		0 & 0 & 0 &  & \\	
		\vdots&  &  &\ddots  & \vdots\\	
		0 & &  & \dots & 0 \\	
		\end{array} \right)}.
	\end{equation}
	As we are taking the determinant of $J_{\varphi_h}$ we can perform elementary row and column operations to manipulate this matrix. Denoting by $\br_i$ and $\bc_i$ the $i$th row and column of $J_{\varphi_h}$ then letting $\bc_1 \rightarrow \sum \bc_i $, then $\br_2\rightarrow \br_2 - \br_1$, we get for row two $\br_2 = (0,\tilde{c}(\bx),0,...,0)$. Then by Laplace expansion across this row, we see that $\det(J_{\varphi_h})=\tilde{c}(\bx)\sum M_{2,j}$, where $M_{2,j}$ are the determinantal  minors across the second row. 
\end{proof}

We also remark that theorem \ref{thm:det} implies that if a Runge-Kutta map has $m$ second integrals corresponding to the cofactor $\tilde{c}(\bx)$, then $\tilde{c}(\bx)^{m-1}$ will divide the Jacobian. Furthermore, for the proof to hold, we only need the map to be locally invertible, such that $\det(J_{\varphi_h})\ne 0$ and therefore should hold for other affinely equivalent maps that possess a rational integral. \\

Using theorem \ref{thm:det} one can easily detect the existence of a rational integral of the form given in corollary \ref{cor:rational integral} by factoring the Jacobian determinant of any Runge-Kutta map applied to an ODE. This immediately provides the form of the cofactor, if it exists as given by the following example.  
\begin{example}[A rational ODE with an unknown rational integral]
	Consider the following ODE 
	\begin{align}\label{ratode}
		\dot{x} =& x-y+z,\\
		\dot{y} =& {\frac {2\,{y}^{2}-xz-{z}^{2}}{y+z}},\\
		\dot{z} =& {\frac {{z}^{2}+z \left( x+y \right) -{y}^{2}}{y+z}}.
	\end{align}
	It is not immediately obvious if this ODE possesses a rational integral. However, we apply to this ODE the forward Euler method, denoted by $\varphi_h(\bx)$, and take its Jacobian determinant. This yields 
	\begin{equation}
		\det\left(J_{\varphi_h}\right) = \frac{K_1K_2}{D^2}
	\end{equation}
	where 
	\begin{align}
		K_1 = & yh+y+z,\\
		K_2 = &   \left( x+2\,y+z \right) {h}^{2}+ \left( x+4\,y+3\,z \right) h+y+z, \\
		D =& y+z.
	\end{align}
	Using $ K_1/D$ as a discrete cofactor and setting  $p(\bx)= \alpha_1 x + \alpha_2 y + \alpha_3 z + \alpha_4$ we solve
	\begin{equation}
		p(\varphi_h(\bx)) = \frac{K_1}{D}p(\bx)
	\end{equation}
	for $\alpha_i$, which gives the following two discrete second integral solutions
	\begin{equation}
		p_1(\bx) = y+z,\quad\text{and}\quad p_2(\bx) = x+y.
	\end{equation}
	This implies that $H(\bx)=(y+z)/(x+y)$ is a first integral of $\varphi_h(\bx)$ and therefore also of the original ODE. Computing
	\begin{equation}
		c(\bx) = \lim\limits_{h\rightarrow0}\left(\frac{\tilde{c}(\bx)-1}{h}\right) = \frac{y}{y+z},
	\end{equation}
	yields the form of the rational (continuous) cofactor. 
\end{example}
We remark that the above method for finding rational affine integrals involves only solving linear systems and hence is fast and scalable to ODE in higher dimensions with free parameters. 

\subsection{Constant cofactor case}

We now restrict our discussion to the special case of affine Darboux polynomials with constant cofactor
\begin{equation}\label{lindp}
\dot{p}(\bx) = \lambda p(\bx),
\end{equation}
where $\lambda$ is constant. If such a Darboux polynomial exists then we can solve for $p(\bx)$ 
\begin{equation}
p(\bx)=Ke^{\lambda t},
\end{equation} 
for some arbitrary constant $K$, and therefore
\begin{equation}
H=p(\bx)e^{-\lambda t}
\end{equation} 
is a time-dependent integral of the ODE \eqref{ODE1}. This has a natural discrete-time analogue as follows. 

\begin{proposition}
	Consider a map $\varphi_h$ with a discrete Darboux polynomial satisfying
	$$p(\varphi_h(\bx))=\tilde{c}p(\bx),$$ where $\tilde{c}$ is constant. This implies that
	\begin{equation}
		p(\bx)=K \tilde{c}^k,
	\end{equation} 
	where $k$ is the iteration index and $K$ is an arbitrary constant, moreover the function 
	\begin{equation}
		\tilde{H} = \tilde{c}^{-k}p(\bx),
	\end{equation}
	is an integral of the map $\varphi_h^{\circ k}$.
\end{proposition}  
We now consider the case where the map $\varphi_h$ comes about as a Runge-Kutta method applied to an ODE that possesses one or more Darboux polynomials of the form \eqref{lindp}. We begin with an example of a Lotka-Volterra system with a time-dependent integral and its discretisation under an arbitrary Runge-Kutta method.

\begin{example}[A Lotka-Volterra system with one time-dependent integral]
	Consider the following Lotka-Volterra system
	\begin{equation}\label{LV}
	\frac{\mathrm{d}}{\mathrm{d}t}\left(\begin{array}{c}
	x_1\\x_2\\x_3\\
	\end{array}\right)
	=
	\left(\begin{array}{c}
	x_{{1}} \left( a_{{1}}x_{{2}}-a_{{2}}x_{{3}}+b \right) \\
	x_{{2}} \left( -a_{{1}}x_{{1}}+a_{{3}}x_{{3}}+b \right) \\
	x_{{3}} \left( a_{{2}}x_{{1}}-a_{{3}}x_{{2}}+b \right)\\
	\end{array}\right).
	\end{equation}
	Then $p_{1}(x) = x_1+x_2+x_3$ is a Darboux polynomial corresponding to the cofactor $c_1=b$, hence 
	\begin{equation}
	H=p_{1}(x)e^{-bt}
	\end{equation}
	is a time-dependent integral of the ODE \eqref{LV}. Now consider discretisation of the above ODE by a the forward Euler method $\varphi_h$. Then, $p_{1}(x)$ is a discrete Darboux polynomial of $\varphi_h$ with the cofactor $\tilde{c}_1=1+bh$, hence
	\begin{equation}
	\tilde{H} = \tilde{c}_1^k p_{1}(x),
	\end{equation}
	where $k$ is the iteration index, is an integral of the map $\varphi_h^{\circ k}$. 
\end{example}

Note that $\tilde{c}_1$ in the above example also corresponds to the stability function $R(bh)$ of the forward Euler method. This is no coincidence and is due to the following corollary of theorem \ref{thm: lin RK cofactors}. 

\begin{corollary}\label{thm:stabfun1}
	If an autonomous polynomial ODE $\dot{\bx}=\bff(\bx)$ possesses an affine Darboux polynomial $p(\bx)$ with a constant cofactor $\lambda$, then a Runge-Kutta map $\varphi_h$ with stability function $R(z)$ satisfies
	\begin{equation}
	p(\varphi_h(\bx))=R(\lambda h)p(\bx),
	\end{equation}
	where 
	\begin{equation}\label{stabfun}
	R(z) = 1 + z\bb^T(I-z\mcA)^{-1}\mathbb{1}
	\end{equation}
	is the (constant) discrete cofactor, $\mathbb{1}$ is the ones vector and $I$ is the $s\times s$ identity matrix.
\end{corollary} 
\begin{proof}
	This can be seen by setting $c(\bx) = \lambda$ in theorem \ref{thm: lin RK cofactors}. Alternatively, we note that as Runge-Kutta methods are affinely equivariant \cite{mclachlan1998six}, we can take $p(x) = x_1$ without loss of generality, then equation \eqref{lindp} is identical to Dahlquist's famous test ODE \cite{wanner1996solving} and $R(z)$ is identical to the the stability function in the context of A-stability of one-step methods. The rest follows from e.g., \cite[Proposition 3.1]{wanner1996solving}. See appendix \ref{A} for details.  
\end{proof}

The fact that the cofactor is the stability function is intuitive. Equation \eqref{DP} implies that $p(\bx) = A_0\exp(\int c(\bx) \dd t)$ for some integration constant $A_0$. Letting $c(\bx)=\lambda$  and $p(\bx)=x_1$ implies that $x_1=A_0\exp(\lambda t)$ as expected for Dalquist's linear test equation, while an order-$p$ Runge-Kutta map yields $\varphi_h(x_1) = R(\lambda h) x_1$, where $R(z)\approx \exp(z)$ is an order-$p$ Pad{\'e} approximation to the exponential function. \\

This result implies that Runge-Kutta methods can generally preserve integrals that are the product of Darboux polynomials. This can be seen by considering an ODE that has the first integral $H(\bx)=p_1(\bx)^{\alpha_1}p_2(\bx)^{\alpha_2}$ where the $p_i(\bx)$ are Darboux polynomials with constant cofactors $\lambda_i$. The fact that $H(\bx)$ is a first integral implies that $\alpha_1\lambda_1+\alpha_2\lambda_2=0$. Then $H(\bx)$ is only a first integral of the Runge-Kutta map $\varphi_h$ if and only if $R(\lambda_1)^{\alpha_1} R(\lambda_2)^{\alpha_2} = 1$, which is only true only when $R(z)=\exp(z)$, which is impossible. The exception of course, is when $\alpha_1=-\alpha_2=\pm1$ and $\lambda_1=\lambda_2$. This results in the integral $H(\bx)=p_1(\bx)/p_2(\bx)$ being preserved by all Runge-Kutta maps, as shown in corollary \ref{cor:rational integral}. However, despite this result, there are some instances where Runge-Kutta methods can preserve a modified integral as we will now discuss. \\

The existence of Darboux polynomials with constant cofactors implies the existence of time-dependent integrals of an ODE and therefore iteration-index-dependent integrals of a Runge-Kutta map in the discrete-time case. One can eliminate this time dependence (and iteration index dependence) by taking quotients. This leads to the preservation of some non-rational modified integrals by Runge-Kutta methods by the following corollary. 

 \begin{corollary}\label{cor:nonratintegrals}
 	Given an ODE with a (non-rational) first integral given by
	\begin{equation}
 	H(\bx) = \frac{p_1(\bx)^{\sigma}}{p_2(\bx)},
 	\end{equation}
 	 where $p_1(\bx)$ and $p_2(\bx)$ are affine Darboux polynomials with constant cofactors $c_1$ and $c_2$ and 
 	\begin{equation}
 	\sigma = \frac{c_2}{c_1},
 	\end{equation}
 	then any Runge-Kutta map $\varphi_h$ with stability function $R(z)$ preserves the modified integral
 	\begin{equation}
 	\tilde{H}(\bx) = \frac{p_1(\bx)^{\tilde{\sigma}}}{p_2(\bx)}
 	\end{equation}
 	with the modified exponent
 	\begin{equation}
 	\tilde{\sigma} = \frac{\ln(R(hc_2))}{\ln(R(hc_1))}.
 	\end{equation}
 \end{corollary}
 \begin{proof}
 	According to corollary \ref{thm:stabfun1}, $\varphi_h$ preserves the Darboux polynomials  $p_1(\bx)$ and $p_2(\bx)$ with the modified cofactors $R(hc_1)$ and $R(hc_2)$. It follows that $\tilde{H}(\bx)$ is also a Darboux polynomial of $\varphi_h$ with cofactor 1, and hence is a first integral of $\varphi_h$. 
 \end{proof}
 This is demonstrated by the following example. 
 \begin{example}[An ODE with a non-rational integral]\label{eg:nonrat}
 	Consider the following ODE in three dimensions
 	\begin{equation}
 	\frac{\mathrm{d}}{\mathrm{d}t}\left(\begin{array}{c}
 	x_1\\x_2\\x_3\\
 	\end{array}\right)
 	=
 	\left(\begin{array}{c}
 	x_{{1}}x_{{2}}+x_{{2}}x_{{3}}+x_{{1}}\\-x_{{1}}x_{{2}}-x_{{2}}x_{{3}}+x_{{2}}\\\left( \sigma-1 \right)  \left( x_{{1}}+x_{{2}} \right) +\sigma\,x_{{
 			3}}\\
 	\end{array}\right),
 	\end{equation}
 	where $\sigma\in\mathbb{R}$. Discretising the above ODE by the forward Euler  method $\varphi_h$ with stability function $R(z) = 1+z$, we find the following affine Darboux polynomials and their corresponding discrete constant cofactors 
 	\begin{gather}
 		p_1(\bx) = x_1+x_2,\quad \tilde{c}_1 = 1 + h,\\
 		p_2(\bx) = x_1+x_2+x_3,\quad \tilde{c}_2 = 1 + \sigma h.
 	\end{gather}
 	Therefore 
 	\begin{equation}
 	\tilde{H} = \frac{(x_1+x_2)^{{\tilde{\sigma}}}}{x_1+x_2+x_3}
 	\end{equation}
 	defines an $h$-dependent integral of $\varphi_h$ with 
 	\begin{equation}
 	\tilde{\sigma} = \frac{\ln\left(1+\sigma h\right)}{\ln\left(1+h\right)}.
 	\end{equation}
 	We can take the continuum limit $\lim\limits_{h\rightarrow0}\left(\tilde{\sigma}\right)=\sigma$, which implies that $\lim\limits_{h\rightarrow0}\tilde{H}(x)=H(x)$, where
 	\begin{equation}
 	H(x)=\frac{(x_1+x_2)^{\sigma}}{x_1+x_2+x_3}.
 	\end{equation}
 	Indeed, the irrational integral $H(x)$ is preserved by the flow of the original ODE. 
 \end{example}
 
\section{ODEs with affine higher integrals}
The notion of first, second, third\footnote{A third integral is a function $K(\bx)$ that is preserved on a particular level set of a first integral $H(\bx)$, e.g., $\dot{K}(\bx) = c(\bx) H(\bx)$} and higher integrals can be generalized by considering solutions of the linear system for
$\underline{{\bp}}(\bx) = (p_1(\bx), p_2(\bx), ..., p_m(\bx))^T\in\mathbb{R}^m$ satisfying \cite{goriely2001integrability}
\begin{equation}
\dot{\underline{\bp}}(\bx) = L(\bx)\underline{\bp}(\bx),
\end{equation}
where $L(\bx)\in\mathbb{R}^{n\times n}$. If $L_{ij}=0$ for $j=1,...,n$ then $p_i(\bx)$ is a first integral and is preserved from arbitrary initial conditions. If $L_{ii}\ne 0$ and $L_{ij}=0$ for $j\ne i$ then $p_i(\bx)$ is a Darboux polynomial with cofactor $L_{ii}$ and is preserved for initial conditions that begin on its zero level sets. If, for example, $L_{ik}\ne 0$ and $L_{ij}=0$ for $j\ne i,k$ then $p_{i}(\bx)$ is preserved for initial conditions starting on the zero level sets of $p_k(\bx)$. Here, if $p_k(\bx)$ is a first integral, then $p_i(\bx)$ is called a third integral. In this sense, one can define the notion of higher integrals for ODE systems. In this section we consider Runge-Kutta methods applied to ODEs with a sub-system of affine higher integrals with constant coefficient matrix $L$. 

 \begin{corollary}\label{thm:Ap}
 	Consider an autonomous ODE in $n$ dimensions that possesses a system of $m\le n$ linearly-independent affine polynomials $\underline{{\bp}}(\bx) = (p_1(\bx), p_2(\bx), ..., p_m(\bx))^T\in\mathbb{R}^m$ satisfying 
 	\begin{equation}\label{ldp}
 	\dot{\underline{\bp}}(\bx) = L\underline{\bp}(\bx),
 	\end{equation}
 	where $L$ is an $m\times m$ matrix that is independent of $\bx$. Then a Runge-Kutta map $\varphi_h$ satisfies
 	\begin{equation}
 	\underline{\bp}(\varphi_h(\bx))=R(hL)\underline{\bp}(\bx),
 	\end{equation}
 	where $R(z)$ is the stability function of $\varphi_h$. 
 \end{corollary}
 \begin{proof}
 	Let $\underline{\bp}(\bx)=D\bx$, where $D$ is an $m\times n$ matrix of rank $m\le n$. The numerical solution of $\varphi_h$ applied to a linear ODE in $n$ dimensions $\dot{\underline{\bp}}=L\underline{\bp}$ is given by (see for example, \cite[p. 194]{hairer2006geometric})
 	\begin{equation}
 	\varphi_h(\underline{\bp}) = R(hL) \underline{\bp}.
 	\end{equation}
 	As Runge-Kutta methods commute with linear transformations we have $\varphi_h(\underline{\bp}(\bx))=\underline{\bp}(\varphi_h(\bx))$, which yields the desired result. 
 \end{proof}
 We remark that an ODE in $n$ dimensions cannot possess more than $n$ functionally independent affine Darboux polynomials with constant cofactor. We demonstrate corollary \ref{thm:Ap} in an example where $L$ is given in Jordan form. But first we will briefly introduce a particular class of Runge-Kutta methods called \textit{diagonal Pad{\'e} Runge-Kutta methods}. 
 \begin{definition}\label{diagpade}
 	A \textit{diagonal Pad{\'e} Runge-Kutta map} is an $s$-stage Runge-Kutta map $\varphi_h$ whose stability function $R(z) = \frac{P(z)}{Q(z)}$ has equal degree in the numerator and denominator, that is $\deg(P(z)) = \deg(Q(z))$. Such a stability function is an order $s$ approximation to $e^{z}$ with numerator and denominator given by
 	\begin{align}
 	P(z) &= \sum_{i=0}^{s}a_i z^i \\
 	Q(z) &= \sum_{i=0}^{s}(-1)^i a_i z^i = P(-z),
 	\end{align}
 	where $a_i$ are the constants
 	\begin{equation}
 	a_i = \frac{s!(2s-i)!}{(2s)!i!(s-i)!}.
 	\end{equation}
 	See, for example, \cite{reusch1988diagonal} and references therein. Such a map has a stability function that satisfies $R(-z)R(z)=1$.
 \end{definition}
 As an example of some well known Diagonal Pad{\'e} Runge-Kutta methods, consider the Runge-Kutta map $\varphi_h(\bx)=\bx'$ defined by 
 \begin{equation}\label{RKa}
 \frac{\bx'-\bx}{h} = (1-2\theta)\bff(\frac{\bx+\bx'}{2}) + \theta \bff(\bx') + \theta\bff(\bx)
 \end{equation}
 then its stability function is the diagonal Pad{\'e} approximation 
 \begin{equation}\label{diag pade O2}
 R(\lambda h) = \frac{1+\frac{1}{2}\lambda h}{1-\frac{1}{2}\lambda h}. 
 \end{equation}
 Note that the three cases $\theta=0$, $\theta=\frac{1}{2}$ and $\theta=-\frac{1}{2}$ respectively correspond to the midpoint rule, trapezoidal rule and Kahan's method, the latter being when $\bff(\bx)$ is quadratic \cite{celledoni2012geometric}. \\
 
 \begin{example}[A 3D ODE with a 2D linear subsystem in Jordan form]
 	Consider the following quadratic ODE in 3 dimensions that was considered in example \eqref{eg:nonrat}
 	\begin{equation}
 	\frac{\mathrm{d}}{\mathrm{d}t}\left(\begin{array}{c}
 	x_1\\x_2\\x_3\\
 	\end{array}\right)
 	=
 	\left(\begin{array}{c}
 	x_{{1}}x_{{2}}+x_{{2}}x_{{3}}+\sigma x_{{2}}\\-x_{{1}}x_{{2}}-x_{{2}}x_{{3}}+\sigma x_{{1}}\\ \left( x_{{1}}+x_{{2}} \right) +\sigma\,x_{{
 			3}}\\
 	\end{array}\right).
 	\end{equation}
 	Setting $\underline{\bp}(\bx)=(x_1+x_2+x_3,x_1+x_2)^T$ then the above ODE has the following linear subsystem
 	\begin{equation}
 	\dot{\underline{\bp}}(\bx) =\left( \begin{array}{cc}
 	\sigma & 1 \\
 	0 & \sigma \\
 	\end{array} \right) \underline{\bp}(\bx) := L \underline{\bp}(\bx).
 	\end{equation}
 	If we apply Kahan's method then we get
 	\begin{align}
 	\underline{\bp}(\varphi_h(\bx)) = & \left( \begin {array}{cc} {\frac {-h\sigma-2}{h\sigma-2}}&4\,{\frac {
 			h}{ \left( h\sigma-2 \right) ^{2}}}\\ \noalign{\medskip}0&{\frac {-h
 			\sigma-2}{h\sigma-2}}\end {array}
 	\right) \underline{\bp}(\bx)\\
 	=& \left(\begin{array}{cc}
 	R(h\sigma) & R'(h\sigma) \\
 	0 & R(h\sigma) \\
 	\end{array} \right) \underline{\bp}(\bx)= R(hL) \underline{\bp}(\bx),
 	\end{align} 
 	which is the form prescribed in corollary \ref{thm:Ap}.
 	
 \end{example}
 
 It is known that when the Kahan map is applied to a Hamiltonian ODE that it preserves a modified Hamiltonian \cite{celledoni2012geometric}. Here we show that there exists some cases where the Kahan map (as well as other diagonal Pad{\'e} Runge-Kutta maps) preserves a quadratic first integral exactly.
 
 \begin{example}[An ODE with a quadratic integral that is preserved exactly]
 	Consider the quadratic ODE in three dimensions
 	\begin{equation}
 	\frac{\mathrm{d}}{\mathrm{d}t}\left(\begin{array}{c}
 	x_1\\x_2\\x_3\\
 	\end{array}\right)
 	=
 	\left(\begin{array}{c}
 	2\,x_{{2}}+x_{{1}}+x_{{3}}+x_{{3}} \left( x_{{1}}+x_{{2}} \right) +{x_
 		{{1}}}^{2} \\
 	-x_{{1}}-x_{{2}}-x_{{3}} \left( x_{{1}}+x_{{2}} \right) -{x_{{1}}}^{2}\\
 	x_{{3}} \left( x_{{1}}+x_{{2}} \right) +{x_{{1}}}^{2}
 	\end{array}\right)	
 	\end{equation}
 	Setting $\underline{\bp}(\bx)=(x_1+x_2,x_2+x_3)^T$ then this satisfies
 	\begin{equation}
 	\dot{\underline{\bp}}(\bx) =\left( \begin{array}{cc}
 	0& 1 \\
 	-1 & 0 \\
 	\end{array} \right) \underline{\bp}(\bx)
 	\end{equation}
 	If we apply the Kahan discretisation to the above ODE we find that $p_{1}(\bx) = \|\underline{\bp}(\bx)\|^2=(x_1+x_2)^2+(x_2+x_3)^2$ is a quadratic Darboux polynomial of $\varphi_h$ with cofactor $\tilde{c}=1$. Therefore $H(\bx)= \|\underline{\bp}(\bx)\|^2$ is an integral of the map, i.e., $H(\varphi_h(\bx))=H(\bx)$. Moreover, we find that any diagonal Pad{\'e} Runge-Kutta map preserves this integral. This is due to the following corollary. 
 \end{example}
 
 \begin{corollary}\label{thm:skewA}
 	Given an ODE $\dot{\bx}=\bff(\bx)$, if there exists $m$ linear polynomials $\underline{\bp}(x)=(p_1(x),p_2(x),...,p_m(x))^T$ that satisfy \begin{equation}
 	\dot{\underline{\bp}} = AS\underline{\bp}(\bx), \quad\mathrm{where}\quad A=-A^T\quad\mathrm{and}\quad S=S^T \label{skewLinODE}
 	\end{equation}
 	then $H(x)=\underline{\bp}(\bx)^TS\underline{\bp}(\bx)$ is a first integral of the ODE and any diagonal Pad{\'e} Runge-Kutta map (e.g., one of the form \eqref{RKa}) preserves this integral exactly. 
 \end{corollary}
 \begin{proof}
 	Using the fact that $R(z)=P(z)P(-z)^{-1}=P(-z)^{-1}P(z)$, $P(hAS)^T=P(-hSA)$ and $P(hSA)S=S\,P(hAS)$ then it's straight forward to show that $H$ is preserved under $\varphi_h$
 	\begin{align}
 	H(\varphi_h(\bx))  = &	\underline{\bp}(\bx)^TR(hAS)^TSR(hAS)\underline{\bp}(\bx) \\
 	= & \underline{\bp}(\bx)^TP(-hAS)^{-T}P(hAS)^TSP(hAS)P(-hAS)^{-1}\underline{\bp}(\bx) \\
 	= & \underline{\bp}(\bx)^TP(hSA)^{-1}P(-hSA)SP(hAS)P(-hAS)^{-1}\underline{\bp}(\bx) \\
 	= & \underline{\bp}(\bx)^TSP(hAS)^{-1}P(-hAS)P(-hAS)^{-1}P(hAS)\underline{\bp}(\bx) \\
 	= & \underline{\bp}(\bx)^TS\underline{\bp}(\bx) \\
 	= & H(\bx)
 	\end{align}
 \end{proof}
 We can therefore make the following statement about linear ODEs with quadratic integrals. 
 \begin{corollary}
 	For all linear ODEs with quadratic first integrals, all diagonal Pad{\'e} Runge-Kutta maps preserve the integral exactly. \\
 \end{corollary}
 \begin{proof}
 	A linear ODE with a quadratic integral $H=\frac{1}{2}\bx^TS\bx$ can be written in the form 
 	\begin{equation}
 	\dot{\bx} = A\nabla H = AS\bx
 	\end{equation}
 	for $A=-A^T$ and $S=S^T$, which is in the form of \eqref{skewLinODE} and according to corollary \ref{thm:skewA}, all diagonal Pad{\'e} Runge-Kutta maps preserve the integral $H$
 \end{proof}
 This is a slight generalisation of proposition 5 in \cite{celledoni2009energy}. \\
 
 Corollary \ref{thm:skewA} also applies to ODEs with multiple quadratic integrals as shown in the following example. 
 \begin{example}[A 5D system with 2 quadratic integrals preserved exactly]
 	Consider the following ODE
 	{\footnotesize
 		\begin{equation}
 		\frac{\mathrm{d}}{\mathrm{d}t}\left(\begin{array}{c}
 		x_1\\x_2\\x_3\\x_4\\x_5\\
 		\end{array}\right)
 		= \left(\begin{array}{c}
 		{x_{{3}}}^{2}+2\,x_{{3}}x_{{4}}+2\,x_{{3}}x_{{5}}+2\,{x_{{4}}}^{2}+2\,
 		x_{{4}}x_{{5}}+2\,{x_{{5}}}^{2}-10\,x_{{1}}-17\,x_{{2}}-7\,x_{{3}}-5\,
 		x_{{4}}+4\,x_{{5}}
 		
 		\\-{x_{{3}}}^{2}-2\,x_{{3}}x_{{4}}-2\,x_{{3}}x_{{5}}-2\,{x_{{4}}}^{2}-2
 		\,x_{{4}}x_{{5}}-2\,{x_{{5}}}^{2}+6\,x_{{1}}+11\,x_{{2}}+5\,x_{{3}}+5
 		\,x_{{4}}-3\,x_{{5}}
 		
 		\\{x_{{3}}}^{2}+2\,x_{{3}}x_{{4}}+2\,x_{{3}}x_{{5}}+2\,{x_{{4}}}^{2}+2\,
 		x_{{4}}x_{{5}}+2\,{x_{{5}}}^{2}-x_{{2}}-x_{{3}}-4\,x_{{4}}+2\,x_{{5}}
 		
 		\\-x_{{1}}-2\,x_{{2}}-x_{{3}}+2\,x_{{4}}
 		
 		\\x_{{1}}+3\,x_{{2}}+2\,x_{{3}}+2\,x_{{4}}-2\,x_{{5}}
 		
 		\\
 		\end{array}\right)
 		\end{equation}}
 Let $\lambda = 2+i$, where $i^2=-1$ and discretise the ODE using Kahan's method, which is a diagonal Pad{\'e} Runge-Kutta map. We find the following affine discrete Darboux polynomials with constant cofactor.
 	\begin{equation}
 	\begin{array}{c|c|c}
 	j&p_j(x) & \tilde{c}_j \\
 	\hline 
 	1&-ix_{{4}}+x_{{1}}+2\,x_{{2}}+x_{{3}} &R(\lambda h)
 	\\
 	2&\left( 106+52\,i \right) x_{{2}}+41\,x_{{3}}- \left( 52-24\,i
 	\right) x_{{5}}+ \left( 65+52\,i \right) x_{{1}}+ \left( 12-15\,i
 	\right) x_{{4}}&R(-\lambda h)
 	\\
 	3&ix_{{4}}+x_{{1}}+2\,x_{{2}}+x_{{3}}&R(\lambda^* h)
 	\\
 	4&\left( 106-52\,i \right) x_{{2}}+41\,x_{{3}}- \left( 52+24\,i
 	\right) x_{{5}}+ \left( 65-52\,i \right) x_{{1}}+ \left( 12+15\,i
 	\right) x_{{4}}&R(-\lambda^* h)
 	\\
 	\end{array}
 	\end{equation}
 	From the above four linear Darboux polynomials, we can construct the following complex quadratic integrals
 	\begin{equation}
 	K_1 = p_1(\bx)p_2(\bx),\quad K_2 = p_3(\bx)p_4(\bx),
 	\end{equation}
 	however if we observe that $K_1=K_2^*$ then the following two independent real integrals can be constructed 
 	\begin{equation}
 	H_1 = K_1+K_2,\quad H_2 = i(K_1-K_2).
 	\end{equation}
 	As these two integrals are independent of $h$, it follows that the ODE also possesses these integrals. 
 \end{example}

 \subsection{Detecting the existence of affine higher integrals}
 Many of our results in this section so far rely on the fact that there exist a change of coordinates $\underline{\bp}(x)=Dx$ that allow us to find a linear ODE subsystem of higher integrals of the form $\dot{\underline{\bp}} = L\underline{\bp}$. A logical question to ask is what is the most general class of ODEs where this sub-system exists and how to calculate this transformation. This is now addressed.
 \begin{theorem}
 	If an ODE in $n$ dimensions can be expressed in the following form
 	\begin{equation}\label{genODE}
 	\dot{\bx} = A \bx + \sum_{i=1}^{k}b_i(\bx) \bv_i
 	\end{equation}
 	for the scalar functions $b_i(\bx)$ and the matrix $V:=[\bv_1,...,\bv_k]\in\mathbb{R}^{n\times k}$ has rank $n-m$, where $m<n$, then there exists the linear transformations $\underline{\bp}=Q\bx\in \mathbb{R}^m$ and $\by=R\bx\in \mathbb{R}^{n-m}$ that decouple the ODE into an $m$ dimensional linear ODE for $\underline{\bp}$ and an $n-m$ dimensional non-linear ODE for $\by$ 
 	\begin{align}
 	\dot{\underline{\bp}} =& L\underline{\bp} \label{decoup}\\
 	\dot{\by} =& \bg(\underline{\bp},\by) \label{decoup2}\\
 	\end{align} 
 	for some matrix $L$. 
 \end{theorem}
 \begin{proof}
 	Given an ODE in the form \eqref{genODE}, then choose the linear transformation  $\underline{\bp}=Q\bx$ such that $\bv_i\in\ker(Q)$. Multiplying the ODE \eqref{genODE} by $Q$ gives 
 	\begin{equation}
 	Q \dot{\bx} = QA\bx +\sum_{i=1}^{k}b_i(\bx) Q\bv_i = QA\bx = \dot{\underline{\bp}},
 	\end{equation}
 	which gives the form of the decoupled linear part of the ODE \eqref{decoup}. Now we need to construct the $n-m$ non-linear part of the ODE \eqref{decoup2}. To do this, we simply choose some matrix $R\in\mathbb{R}^{(n-m)\times n}$ whose rows are independent to the rows in $Q$. This then defines the one-to-one transformation 
 	\begin{equation}
 	\bz := \left(\begin{array}{c}
 	\underline{\bp}\\\by\\
 	\end{array}\right) = \left(\begin{array}{c}
 	Q\\R\\
 	\end{array}\right) \bx := G. \bx
 	\end{equation}
 	Then the transformed ODE $$\dot{\bz} = G\bff(\bz)$$ is in the desired form. 
 \end{proof}
 
 
 A logical question now is how does one find such a linear transformation. We will address this now by presenting a systematic algorithm to transform an ODE given in the form \eqref{genODE} into its decoupled form \eqref{decoup} and \eqref{decoup2}. 
 \begin{algorithm}\label{alg}
 	Given an ODE $\dot{\bx} = \bff(\bx)$
 	\begin{enumerate}
 		\item Let $i=1$. 
 		\item Solve the condition $\nabla p_i(\bx)^T \bff(\bx) = \lambda p_i(\bx)$ for affine $p_i(\bx)$ and constant $\lambda_i$. 
 		\subitem (a) If there is no solutions, or they have all been used in successive steps, then end the algorithm. 
 		\subitem (b) If there are one or more solutions, pick one, set $i\rightarrow i+1$ and move to the next step. 
 		\item Set $p_{i-1}(\bx)=0$ in $\bff$ and solve the condition $\left. \nabla p_{i}(\bx)^T \bff(\bx)\right|_{p_{i-1}(\bx)=0} = \lambda_{i} p_{i}(\bx)$ for $p_{i}(\bx)$ and $\lambda_{i}$. 
 		\subitem (a) If there is a solution, then set $i\rightarrow i+1$ and repeat step 3
 		\subitem (b) If there is no solution, go back to step 2 and pick a different Darboux polynomial solution 
 		\item Calculate $Q = \nabla \underline{\bp}(\bx)$
 		\item Calculate $L$ from $\nabla(Q.f) = LQ $
 		\item Choose a matrix $R$ s.t. 
 		\begin{equation}
 		G=\left(\begin{array}{c}
 		Q \\ R \\
 		\end{array}\right)
 		\end{equation}
 		has full rank. Then the ODE
 		\begin{equation}
 		\dot{\bz} = G\bff(\bz)
 		\end{equation}
 		is in the desired form. End algorithm. 
 	\end{enumerate}
 \end{algorithm}
 We will now demonstrate this algorithm with an example. 
 \begin{example}[Example of algorithm \ref{alg}] 
 	
 	Consider the following vector field in five dimensions 
 	\begin{equation}
 	\frac{\mathrm{d}}{\mathrm{d}t}\left(\begin{array}{c}
 	x_1\\x_2\\x_3\\x_4\\x_5\\
 	\end{array}\right)
 	=
 	\left( 
 	\begin{array}{c}
 	-2\,x_{{1}}x_{{3}}+5\,{x_{{2}}}^{2}+{x_{{4}}}^{2}-2\,{x_{{5}}}^{2}+x_{
 		{2}}+x_{{5}}
 	\\
 	\left( -2\,x_{{3}}+15 \right) x_{{1}}+5\,{x_{{2}}}^{2}+{x_{{4}}}^{2}-
 	2\,{x_{{5}}}^{2}+6\,x_{{2}}+12\,x_{{3}}+8\,x_{{4}}+x_{{5}}
 	\\
 	\left( 2\,x_{{3}}+2 \right) x_{{1}}-5\,{x_{{2}}}^{2}-{x_{{4}}}^{2}+2
 	\,{x_{{5}}}^{2}+x_{{3}}+2\,x_{{4}}-x_{{5}}
 	\\
 	\left( 2\,x_{{3}}-7 \right) x_{{1}}-5\,{x_{{2}}}^{2}-{x_{{4}}}^{2}+2
 	\,{x_{{5}}}^{2}-3\,x_{{2}}-6\,x_{{3}}-3\,x_{{4}}-x_{{5}}
 	\\
 	x_{{1}}x_{{3}}-2\,{x_{{2}}}^{2}+{x_{{5}}}^{2}
 	\\	
 	\end{array}
 	\right)
 	\end{equation}
 	We will now implement algorithm 1 to decouple this ODE into a set of linear and non-linear equations. First look for an affine Darboux polynomial that has a constant cofactor. We find the following Darboux polynomial
 	\begin{equation}
 	p_1 = x_{{1}}+x_{{2}}+2\,x_{{4}}
 	\end{equation}
 	with cofactor 1 is the only affine Darboux polynomial that has constant cofactor. Now eliminate a variable from the vector field by setting $p_1=0$, for example, by substituting $x_1 = - x_2 - 2x_4$ into $\bff$. 
 	\begin{equation}
 	\frac{\mathrm{d}}{\mathrm{d}t}\left(\begin{array}{c}
 	x_2\\x_3\\x_4\\x_5\\
 	\end{array}\right)
 	=
 	\left( 
 	\begin{array}{c}
 	\left( -2\,x_{{3}}+15 \right)  \left( -x_{{2}}-2\,x_{{4}} \right) +5
 	\,{x_{{2}}}^{2}+{x_{{4}}}^{2}-2\,{x_{{5}}}^{2}+6\,x_{{2}}+12\,x_{{3}}+
 	8\,x_{{4}}+x_{{5}}
 	\\
 	\left( 2\,x_{{3}}+2 \right)  \left( -x_{{2}}-2\,x_{{4}} \right) -5\,{
 		x_{{2}}}^{2}-{x_{{4}}}^{2}+2\,{x_{{5}}}^{2}+x_{{3}}+2\,x_{{4}}-x_{{5}}
 	\\
 	\left( 2\,x_{{3}}-7 \right)  \left( -x_{{2}}-2\,x_{{4}} \right) -5\,{
 		x_{{2}}}^{2}-{x_{{4}}}^{2}+2\,{x_{{5}}}^{2}-3\,x_{{2}}-6\,x_{{3}}-3\,x
 	_{{4}}-x_{{5}}
 	\\
 	\left( -x_{{2}}-2\,x_{{4}} \right) x_{{3}}-2\,{x_{{2}}}^{2}+{x_{{5}}}
 	^{2}
 	\\	
 	\end{array}
 	\right)
 	\end{equation}
 	We now look for Darboux polynomials with constant cofactor on the resulting four dimensional vector field and find 
 	\begin{equation}
 	p_2 = -x_{{2}}+x_{{3}}-2\,x_{{4}}
 	\end{equation}
 	also with cofactor 1. We now substitute $x_2=x_3-2x_4$ and compute constant cofactor Darboux polynomials on the resulting three dimensional system 
 	\begin{equation}
 	\frac{\mathrm{d}}{\mathrm{d}t}\left(\begin{array}{c}
 	x_3\\x_4\\x_5\\
 	\end{array}\right)
 	=
 	\left( 
 	\begin{array}{c}
 	- \left( 2\,x_{{3}}+2 \right) x_{{3}}-5\, \left( x_{{3}}-2\,x_{{4}}
 	\right) ^{2}-{x_{{4}}}^{2}+2\,{x_{{5}}}^{2}+x_{{3}}+2\,x_{{4}}-x_{{5}
 	}
 	\\
 	- \left( 2\,x_{{3}}-7 \right) x_{{3}}-5\, \left( x_{{3}}-2\,x_{{4}}
 	\right) ^{2}-{x_{{4}}}^{2}+2\,{x_{{5}}}^{2}-9\,x_{{3}}+3\,x_{{4}}-x_{
 		{5}}
 	\\
 	-{x_{{3}}}^{2}-2\, \left( x_{{3}}-2\,x_{{4}} \right) ^{2}+{x_{{5}}}^{2
 	}
 	\\
 	
 	\end{array}
 	\right)
 	\end{equation}
 	this vector field has the polynomial $p_3 = x_4-x_3$ with cofactor $1$. Substituting $x_3=x_4$ into the above three dimensional vector field, we find that there are no more affine Darboux polynomials with constant cofactor. 
 	Setting $\underline{\bp}(\bx) = (p_1(x),p_2(x),p_3(x))^T$, we can write $\underline{\bp}(\bx) = Q \bx$, where 
 	\begin{equation}
 	Q=\left[ \begin {array}{ccccc} 1&1&0&2&0\\ \noalign{\medskip}0&-1&1&-2&0
 	\\ \noalign{\medskip}0&0&-1&1&0\end {array} \right] .
 	\end{equation}
 	It must therefore be possible to write the original ODE as the following decoupled set of ODEs
 	\begin{align}
 	\dot{\underline{\bp}} =& L \underline{\bp}, \\
 	\dot{\by} =& \bg(\by,\underline{\bp}),
 	\end{align}
 	where $L$ is triangular. Multiplying the ODE $\dot{\bx}=\bff(\bx)$ by $Q$, we get the identity $Q\bff(\bx) = L\underline{\bp}$. In other words, $Q\bff(\bx) = B\bx$ must be linear. The coefficients of the matrix $L$ can therefore be easily found by solving the linear problem $B=LQ$. Doing so, we find 
 	\begin{equation}
 	L =  \left[ \begin {array}{ccc} 1&0&0\\ \noalign{\medskip}1&1&0
 	\\ \noalign{\medskip}-9&-6&1\end {array} \right] .
 	\end{equation}
 	We can now define the linear transformation 
 	\begin{equation}
 	\left(\begin{array}{c}
 	p \\ y \\
 	\end{array}\right) = \left[\begin{array}{c}
 	Q \\ R \\
 	\end{array}\right] \bx
 	\end{equation} 
 	where $R$ is any $2\times n$ matrix whose rows are independent to the rows of $Q$. We will choose $R = [0,I_2]$. Hence by the above linear transformation, the ODE reads
 	\begin{align}
 	\dot{\underline{\bp}} &= L \underline{\bp}, \\
 	\dot{y_1} &= 2\,p_{{1}}p_{{3}}-5\,{p_{{2}}}^{2}-{y_{{1}}}^{2}+2\,{y_{{2}}}^{2}-7\,p
 	_{{1}}-3\,p_{{2}}-6\,p_{{3}}-3\,y_{{4}}-y_{{5}},
 	\\
 	\dot{y_2} &=p_{{1}}p_{{3}}-2\,{p_{{2}}}^{2}+{y_{{2}}}^{2}.
 	\end{align}
 \end{example}
 \section{Conclusion}
In this paper we have shown that all Runge-Kutta methods will preserve all affine second integrals of ODEs. This implies that all Runge-Kutta methods will preserve all rational first integrals where the numerator and denominator are affine. Furthermore, we have demonstrated a method for detecting the existence of such integrals in an ODE, if it exists. We have also shown that Runge-Kutta methods cannot preserve irreducible quadratic second integrals, in general. \\

The constant cofactor case is also studied, which gives rise to certain iteration-index dependent first integrals. Due to this, one can eliminate this iteration-index-dependence by taking quotients, which leads to a the preservation of certain modified first integrals by Runge-Kutta methods (corollary \ref{cor:nonratintegrals}). This is the first example where an arbitrary Runge-Kutta method preserves such an integral. For ODE systems that possess affine higher integrals, we show that in some cases, the diagonal Pad{\'e} Runge-Kutta methods can preserve quadratic integrals exactly. Lastly, we present an algorithm that detects the existence of such affine integrals as well as determines the linear transformation that decouples the affine subsystem of higher integrals from the nonlinear components. 
 \section*{Acknowledgments}
 BK Tapley would like to thank GRW Quispel for many useful comments and suggestions, especially with regards to the constant cofactor case.

 \bibliographystyle{unsrt}
 \bibliography{bibliography.bib}
 
 \appendix
 \section{Derivation of stability matrix for corollary \ref{thm:stabfun1}}\label{A}
 Let $\varphi_h$ denote the Runge-Kutta map defined by equations \eqref{butchertable}, \eqref{rk0} and \eqref{rk1}. Now let $G =(g_1,g_2,...,g_s)^T $ denote the $s\times n$ matrix whose $i$'th row is $g_i^T$. Then if $f(x) = \lambda x$ as in the case of our test equation \eqref{rk0} can be written as 
 \begin{align}
 G = \mathbb{1}x^T + h\lambda \mcA G = (I-h\lambda \mcA)^{-1}\mathbb{1}x^T
 \end{align}
 We therefore have 
 \begin{equation}
 g_i=G^T\hat{e_i}=\left(\hat{e_i}^T(I-h\lambda \mcA)^{-1}\mathbb{1}\right) x
 \end{equation} 
 where $\hat{e_i}\in\mathbb{R}^s$ is the $i$'th canonical unit basis vector with a 1 in the $i$'th component and 0 elsewhere. Inserting this into equation \eqref{rk1} gives
 \begin{align}
 \varphi_h(x) = & x + h \lambda \sum_{j=1}^{s} b_{j}\hat{e_i}^T(I-h\lambda \mcA)^{-1}\mathbb{1} x \\
 = & (1 + h \lambda  b^T (I-h\lambda \mcA)^{-1}\mathbb{1}) x\\
 := & R(\lambda h) x
 \end{align}
\end{document}